\documentclass{amsart}
\usepackage{amsmath, amsthm, amscd, amsfonts, amssymb}

\makeatletter \oddsidemargin.9375in \evensidemargin \oddsidemargin
\newtheorem{thm}{Theorem}[section]
\newtheorem{lem}[thm]{Lemma}
\newtheorem{pro}[thm]{Proposition}

\theoremstyle{definition}

\newtheorem{example}[thm]{Example}
\theoremstyle{remark}

\numberwithin{equation}{section}

\begin{document}

\title{Some homological properties of $T$-Lau product algebra}
\author{ N. RAZI AND A. POURABBAS}
\address{Faculty of Mathematics and Computer
Science, Amirkabir University of Technology, 424 Hafez Avenue,
Tehran 15914, Iran}
\email{Razina@aut.ac.ir    }
 \email{arpabbas@aut.ac.ir}



\begin{abstract}
Let $T$ be a homomorphism from a Banach algebra $B$ to a Banach algebra $A$. The Cartesian product space $A\times B$ with $T$-Lau multiplication and $\ell^1$-norm
becomes a new Banach algebra  $A\times _T B$.  
We investigate the notions such as approximate amenability, pseudo amenability, $\phi$-pseudo amenability, $\phi$-biflatness
and $\phi$-biprojectivity for Banach algebra $A\times_T B$.
We also present an example to show that  approximate amenability of 
$A$ and $B$ is not stable for $A\times _TB$.
Finally we 
characterize  the double centralizer algebra of $A\times _T B$ and present an application of this characterization.
\end{abstract}

\maketitle
\section{Introduction and Preliminaries}
Let $A$ and $B$ be Banach algebras and let $T: B\rightarrow A$ be an algebra homomorphism.
Then we consider $A\times B$ with the following product
$$(a,b)\times _ T (c,d)=(ac+T(b)c+aT(d),bd)\hspace*{2cm} ((a,b),(c,d)\in A\times B).$$
The Cartesian product space $A\times B$ with this product is denoted by $A\times _T B$.
Let $A$ and $B$ be Banach algebras and let $\Vert T\Vert\leq 1$. Then we consider $A\times _T B$
with the following norm
$$\Vert (a,b)\Vert=\Vert a\Vert +\Vert b\Vert\hspace*{2cm} ((a,b)\in A\times _T B).$$
We note that $A\times _T B$ is a Banach algebra with this norm.

Suppose that $T:B\rightarrow A$
is an algebra homomorphism with $\Vert T\Vert\leq 1$ and $A$ is a commutative Banach algebra. Then Bhatt and Dabshi \cite{BHA} have studied the properties,  
such as Gelfand space, Arens regularity and amenability of $A\times _T B$. Moreover suppose that $A$ is unital with unit element
$e$ and $\psi_0:B\rightarrow \mathbb{C}$ is a multiplicative 
 linear functional on $B$. If we define $T:B\rightarrow A$ by $T(b)=\psi_0(b)e$, then the product $\times_T$ coincides with the Lau product \cite{LAU}. The group algebra $L^1(G)$, the measure algebra $M(G)$,
the Fourier algebra $A(G)$ of a locally compact group $G$ and the Fourier-Stieltjes algebra
of a topological group are the examples of Lau algebra \cite{LAU}. Lau product was extended by Sangani Monfared for the general case \cite{MON}.  Many basic properties of $A\times _\theta B$ such as existence of a  bounded approximate identity, spectrum, topological center, the ideal structure, biflatness and biprojectivity are investigated in \cite{MON} and \cite{VISH} . 

Following \cite{BHA}, Abtahi  {\it et al.} \cite{ABT} for every Banach algebras $A$ and $B$ defined the Banach algebra $A\times _T B$ equipped with algebra multiplication
$$(a,b)\times _ T (c,d)=(ac+T(b)c+aT(d),bd)\hspace*{2cm} ((a,b),(c,d)\in A\times B)$$
and the norm $\Vert (a,b)\Vert=\Vert a\Vert + \Vert b\Vert$. They studied the homological properties of this Banach algebra such as biflatness,
biprojectivity and existence of a approximate identity.

We recall some basic definition of the homological properties.  
A Banach algebra $A$ is called  biprojective if $\pi_A : A\hat{\otimes}A\longrightarrow A$ has a bounded
right inverse which is an $A-$bimodule map. A Banach algebra $A$ is called biflat if the adjoint map
${\pi_A} ^{*} : A^*\longrightarrow (A\hat{\otimes}A)^*$ of $\pi_A$ has a bounded left inverse which is an $A-$bimodule map. Here 
the product morphism  $\pi_A :A\hat{\otimes}A\longrightarrow A$ for a Banach algebra $A$ is defined by $\pi_A(a\otimes b)=ab$.
It is clear that $\pi_A$ is an $A-$bimodule map.

 A Banach algebra $A$ is called approximately amenable if for every $A$-bimodule $X$, any derivation
$D: A\rightarrow X^*$ is approximately inner.
A Banach algebra $A$ has an approximate identity if there is a net $\lbrace\eta _{\alpha}\rbrace\subseteq A$
such that $\lim \eta _{\alpha}a-a=\lim a\eta _{\alpha}-a=0$ for all $a\in A$.
A Banach algebra $A$ has a weak approximate identity if there is a net $\lbrace\eta _{\alpha}\rbrace\subseteq A$
such that $\lim f(\eta _{\alpha}a -a)=\lim f(a\eta _{\alpha} -a)=0$ for all $a\in A$ and $f\in A^*$.

The notion of pseudo amenability of Banach algebras was introduced by Ghahramani and Zhang in \cite{GHAH}.
A Banach algebra $A$ is said to be pseudo amenable if there exists a net $\lbrace\rho _{\alpha}\rbrace\subseteq A\hat{\otimes}A$,
such that $a\rho _{\alpha}-\rho _{\alpha}a\rightarrow 0$ and $\pi_A (\rho _{\alpha})a\rightarrow a$
for all $a\in A$ \cite{GHAH}.

The notion of character pseudo amenability was introduced by Nasr-Isfahani and Nemati in \cite{NASR}.
Let $A$ be a Banach algebra and let $\phi\in \Delta(A)$, where $\Delta(A)$
is the character space of $A$. We say that $A$ is $\phi$-pseudo amenable if it has a (right)$\phi$-approximate
diagonal,  that is, there exists a not necessarily bounded net $(m_{\alpha})\subseteq A\hat{\otimes}A$ such that
$$\phi(\pi_A(m_{\alpha}))\rightarrow 1\quad
\hbox{and}
\quad\Vert a\cdot m_{\alpha}-\phi(a)m_{\alpha}\Vert\rightarrow 0$$
for all $a\in A$ \cite{NASR}.

 Let $A$ be a Banach algebra and let $\phi\in \Delta (A)$.
 Then $A$ is called $\phi$-biprojective if there exists a bounded $A$-bimodule morphism $\rho: A\rightarrow A\hat{\otimes}A$
 such that $\phi\circ\pi_A\circ\rho(a)=\phi(a)$ for every $a\in A$. 
 A Banach algebra $A$ is called $\phi$-biflat if there exists a bounded $A$-bimodule morphism 
 $\rho_A: A\rightarrow (A\hat{\otimes}A)^{**}$ such that $\tilde{\phi} \circ{\pi_A}^{**}\circ\rho(a)=\phi(a)$ for every $a\in A$, where
 $\tilde{\phi} $ is a unique extension of $\phi$ on $A^{**}$  defined by $\tilde{\phi} (F)=F(\phi)$ for every $F\in A^{**}$. It is 
 clear that this extension remains to be a character on $A^{**}$ \cite{saha}.

The purpose of this paper is
to determine some homological properties of $ A\times _T B$ for every Banach algebras $A$ and $B$, such as approximate amenability, 
pseudo amenability, $\varphi$-pseudo amenability, $\varphi$-biflatness
and $\varphi$-biprojectivity.
After all we characterize  the double centralizer algebra of $A\times _T B$ and we will give  an application of this characterization.
We also present an example which show that
$A$ and $B$ are approximately amenable but $A\times _TB$ is not approximately amenable.

Recall that $A^{*}$, the dual  of a Banach algebra $A$, is a Banach $A-$bimodule with
the module operations defined by $$\left\langle f\cdot a,b\right\rangle=\left\langle f,ab\right\rangle,\quad \left\langle a\cdot f,b\right\rangle=\left\langle f,ba\right\rangle, \qquad (f\in A ^*, a, b\in A).$$

Let $A$ be a Banach algebra. The second dual $A^{**}$ of $A$ with Arens products $\square$ and $\lozenge$ which are defined by 
$\left\langle m\square n,f\right\rangle=\left\langle m,n\cdot f\right\rangle$, where $\left\langle n\cdot f,a\right\rangle=\left\langle n,f\cdot a\right\rangle,$ and similarly
$\left\langle m\lozenge n,g\right\rangle=\left\langle n,f\cdot m\right\rangle$, where $ \left\langle f\cdot m,a\right\rangle=\left\langle m,a\cdot f\right\rangle$ for every $a\in A$, $f\in A^*$ and $m,n\in A^{**}$,
becomes a Banach algebra.

A Banach $A$-bimodule $X$ is called neo-unital if for every $x\in X$
there exist $a,a^{\prime}\in A$ and $y,y^{\prime}\in X$ such that
$ay=x=y^{\prime}a^{\prime}$.

The dual space $(A\times_{T} B)^*$ is identified with $A^* \times B^*$ via
$$\left\langle(f,g),(a,b)\right\rangle=f(a)+g(b),\quad (a\in A, b\in B, f\in A^*, g\in B^*)$$
for more details see\cite[Theorem 1.10.13]{MEG}.
Also the dual space  $(A\times _T B)^*$ is  $(A\times _T B)$-bimodule with the module operations defined by
$$(a,b)\cdot (f,g)=(a\cdot f+ T(b)\cdot f,(a\cdot f)\circ T+b\cdot g)$$
$$(f,g)\cdot (a,b)=(f\cdot a+f\cdot T(b),(f\cdot a)\circ T+g\cdot b).$$
Moreover $A\times_TB$ is a Banach $A$-bimodule under the module actions $a^{\prime}\cdot(a,b)=(a^{\prime},0)\times_T(a,b)$
and $(a,b)\cdot a^{\prime}=(a,b)\times_T(a^{\prime},0)$, for all $a,a^{\prime}\in A, b\in B$.
Similarly $A\times_TB$ is  a Banach $B$-bimodule.

\section{$\phi$-biflat, $\phi$-biprojective}
In this section we investigate $\phi$-biflatness and  $\phi$-biprojectivity of  Banach algebra $A\times_TB$ 

Recall that the character space of  $A\times_TB$ is determined by 
$$\lbrace (\phi,\phi\circ T): \phi\in \Delta(A)\rbrace\cup\lbrace(0,\psi):\psi\in \Delta(B)\rbrace.$$
see  \cite[Theorem 2.1]{BHA} for more details.
\begin{thm}
Let $A,B$ be Banach algebras and let $T: B\rightarrow A$ be an algebra homomorphism with $\Vert T\Vert\leq 1$. Then
\begin{enumerate}
\item[(i)]If $A\times_ T B$ is $(\phi,\phi\circ T)$-biflat for $\phi\in \bigtriangleup(A)$, 
 then $A$
is $\phi$-biflat.
\item[(ii)]If $A\times_ T B$ is $(0,\psi)$-biflat for $\psi\in \bigtriangleup(B)$, 
 then $B$
is $\psi$-biflat.
\end{enumerate}
\begin{proof}
(i) Since $A\times_ T B$ is $(\phi,\phi\circ T)$-biflat, there exists a $A\times_ T B$-bimodule morphism
$\rho_{A\times_ T B}:A\times_ T B\rightarrow ((A\times_ T B)\hat{\otimes} (A\times_ T B))^{**}$ such that
${\pi^{**}_{A\times _TB}}\circ\rho_{A\times_ T B}(a,b)(\phi,\phi\circ T)=\phi(a)+\phi\circ T(b)$ for every $(a,b)\in A\times_ T B.$ We 
define $\rho_A:A\rightarrow (A\hat{\otimes}A)^{**}$
by $\rho_A=(r_A{\otimes}r_A)^{**}\circ\rho_{A\times_ T B}\circ q_A$, 
where $r_A:A\times_ T B\rightarrow A$ is defined by $r_A((a,b))=a+T(b)$ for every $(a,b)\in A\times_ T B$ 
and $q_A:A\rightarrow A\times_ T B$ is defined by $q_A(a)=(a,0)$ for every $a\in A$. We prove that for every $a\in A$ 
\begin{equation}\label{e7}
\pi_A^{**}\circ\rho_A(a)(\phi)=\phi(a).
\end{equation}
To calculate the left hand side of (\ref{e7}) we have
\begin{equation}\label{e8}
\begin{split}
\pi_A^{**}\circ\rho_A(a)(\phi)
=&\pi_A^{**}\circ (r_A{\otimes}r_A)^{**}\circ\rho_{A\times_ T B}\circ q_A(a)(\phi)\\
=&\pi_A^{**}\circ (r_A{\otimes}r_A)^{**}\circ\rho_{A\times_ T B}(a,0)(\phi)\\
=&(r_A{\otimes}r_A)^{**}\circ\rho_{A\times_ T B}(a,0)(\pi_A^{*}(\phi))\\
=&(r_A{\otimes}r_A)^{**}\circ\rho_{A\times_ T B}(a,0)(\phi\circ\pi_A)\\
=&\rho_{A\times_ T B}(a,0)\circ(r_A{\otimes}r_A)^{*}(\phi\circ\pi_A).
\end{split}
\end{equation}
To calculate the left hand side of (\ref{e8}) for every $(x_1,y_1){\otimes}(x_2,y_2)\in (A\times_ T B)\hat{\otimes}(A\times_ T B)$, we have
\begin{equation}\label{e9}
\begin{split}
\rho_{A\times_ T B}(a,0)\circ(r_A{\otimes}r_A)^{*}(\phi\circ\pi_A)&((x_1,y_1)\otimes (x_2,y_2))\\
=&\rho_{A\times_ T B}(a,0)(\phi\circ\pi_A (r_A(x_1,y_1){\otimes}r_A(x_2,y_2))\\
=&\rho_{A\times_ T B}(a,0)(\phi\circ\pi_A(x_1+T(y_1)){\otimes}(x_2+T(y_2))\\
=&\rho_{A\times_ T B}(a,0)(\phi(x_1x_2)+\phi(x_1)\phi\circ T(y_2)\\ &\qquad+\phi\circ T(y_1)\phi(x_2)+\phi\circ T(y_1y_2)).
\end{split}
\end{equation}
But for every $(x_1,y_1){\otimes}(x_2,y_2)\in (A\times_ T B)\hat{\otimes}(A\times_ T B)$, we have
\begin{equation}\label{e10}
\begin{split}
\pi_{A\times_ T B}^*(\phi,\phi\circ T)((x_1,y_1){\otimes}(x_2,y_2))
=&(\phi,\phi\circ T)(x_1x_2+T(y_1)x_2+x_1T(y_2),y_1y_2)\\
=&\phi(x_1x_2)+\phi\circ T(y_1)\phi(x_2)+\phi(x_1)\phi\circ T(y_2)+\phi\circ T(y_1y_2).
\end{split}
\end{equation}
Now If we replace left hand side of (\ref{e10}) in the last line of (\ref{e9}), we obtain
\begin{equation}\label{e11}
\begin{split}
\rho_{A\times_ T B}(a,0)\circ(r_A{\otimes}r_A)^{*}(\phi\circ\pi_A)
=&\rho_{A\times_ T B}(a,0)(\pi_{A\times_ T B}^*(\phi,\phi\circ T))\\
=&\pi_{A\times_ T B}^{**}\circ\rho_{A\times_ T B}(a,0)(\phi,\phi\circ T)\\
=&\phi(a)+\phi\circ T(0)
=\phi(a).
\end{split}
\end{equation}
Comparing (\ref{e8}) and (\ref{e11}) we see that $\pi_A^{**}\circ\rho_A(a)(\phi)=\phi(a)$.
Moreover $\rho_A$ is a $A$-bimodule map, therefore
$A$ is $\phi$-biflat.

(ii) Since $A\times_ T B$ is $(0,\psi)$-biflat, there exists a $A\times_ T B$-bimodule morphism
 $\rho_{A\times_ T B}: A\times_ T B\rightarrow ((A\times_ T B)\hat{\otimes}(A\times_ T B))^{**}$
such that for every $(a,b)\in A\times_ T B$ we have
$$\pi_{A\times_ T B}^{**}\circ \rho_{A\times_ T B}((a,b))(0,\psi)=(0,\psi)(a,b)=\psi(b).$$
Now we define $\rho_B=(p_B\otimes p_B)^{**}\circ\rho_{A\times_ T B}\circ q_B$, where $p_B:A\times_ T B\rightarrow B$
is defined by $p(a,b)=b$ and $q_B:B\rightarrow A\times_ T B$ is defined by $q_B(b)=(0,b)$ for every $(a,b)\in A\times_ T B$. 
So we have to show that  
\begin{equation}\label{e12}
\pi_B^{**}\circ\rho_B(b)(\psi)=\psi(b)\quad(b\in B).
\end{equation}
To calculate the left hand side of (\ref{e12}) we have
\begin{equation}\label{e13}
\begin{split}
\pi_B^{**}\circ\rho_B(b)(\psi)
=&\pi_B^{**}\circ (p_B\otimes p_B)^{**}\circ \rho_{A\times_ T B}(0,b)(\psi)\\
=&(p_B\otimes p_B)^{**}\circ \rho_{A\times_ T B}(0,b)(\pi _B^*(\psi))\\
=&(p_B\otimes p_B)^{**}\circ \rho_{A\times_ T B}(0,b)(\psi\circ\pi_B)\\
=&\rho_{A\times_ T B}(0,b)\circ (p_B\otimes p_B)^{*}(\psi\circ\pi_B).
\end{split}
\end{equation}
To calculate the left hand side of (\ref{e13}) for every $(x_1,y_1){\otimes}(x_2,y_2)\in (A\times_ T B)\hat{\otimes}(A\times_ T B)$, we have
\begin{equation}\label{e14}
\begin{split}
\rho_{A\times_ T B}(0,b)\circ (p_B\otimes p_B)^{*}&(\psi\circ\pi_B)((x_1,y_1)\otimes(x_2,y_2))\\
&=\rho_{A\times_ T B}(0,b)(\psi\circ\pi_B)(p_B\otimes p_B)((x_1,y_1)\otimes(x_2,y_2))\\
&=\rho_{A\times_ T B}(0,b)(\psi\circ\pi_B)(y_1\otimes y_2)\\
&=\rho_{A\times_ T B}(0,b)\psi(y_1y_2).
\end{split}
\end{equation}
But $\psi(y_1y_2)$ in the above equation is obtained from
\begin{equation}\label{e15}
\pi_{A\times_ T B}^*(0,\psi)((x_1,y_1)\otimes(x_2,y_2))
=(0,\psi)(x_1x_2+x_1T(y_2)+T(y_1)x_2, y_1y_2)
=\psi(y_1y_2).
\end{equation}
Now if we replace  (\ref{e15}) in (\ref{e14})  we obtain 
\begin{equation}\label{e17}
\begin{split}
\rho_{A\times_ T B}(0,b)\circ (p_B\otimes p_B)^{*}(\psi\circ\pi_B)
=&\rho_{A\times_ T B}(0,b)\circ \pi_{A\times_ T B}^*(0,\psi)\\
=&\pi_{A\times_ T B}^{**}\circ\rho_{A\times_ T B}((0,b))(0,\psi)\\
=&(0,\psi)(0,b)
=\psi(b).  
\end{split}
\end{equation}
Therefore (\ref{e13}) and (\ref{e17}) follow that $\pi_B^{**}\circ\rho_B(b)(\psi)=\psi(b)$.
Since $\rho_B$ is a $B$-bimodule map,  $B$ is $\psi$-biflat.
\end{proof}
\end{thm}
\begin{thm}
Let $A$ and  $B$ be Banach algebras, let $T: B\rightarrow A$
be an algebra homomorphism with $\parallel T\parallel\leq 1$ and let $\phi\in \bigtriangleup (A)$
and $\psi\in \bigtriangleup (B)$. Then
\begin{enumerate}
\item[(i)]  $A\times _T B$ is $(\phi,\phi\circ T)$-biprojective
if and only if $A$ is $\phi$-biprojective.
\item[(ii)]  $A\times _TB$ is $(0,\psi)$-biprojective if and only if
$B$ is $\psi$-biprojective. 
\end{enumerate}
\begin{proof}
(i) Let $A\times _T B$ be $(\phi,\phi\circ T)$-biprojective. Then there exists a bounded $A\times _T B$-bimodule morphism
$\rho _{A\times _T B}:A\times _T B\rightarrow (A\times _T B)\hat{\otimes}(A\times _T B)$
such that 
\begin{equation}\label{e18}
(\phi ,\phi\circ T)\circ\pi_{A\times _T B}\circ \rho_{A\times _T B}=(\phi ,\phi\circ T).
\end{equation}
To show that $A$ is $\phi$-biprojective, we define $\rho_A: A\rightarrow A\hat{\otimes}A$ by $\rho_A=(r_A\otimes r_A)\circ \rho _{A\times _T B}\circ q_A$, where
$q_A:A\rightarrow A\times _T B$ is defined by $q_A(a)=(a,0)$
for every $a\in A$ and $r_A: A\times _T B\rightarrow A$
is defined by $r_A(a,b)=a+T(b)$, for
every $(a,b)\in A\times _T B$. Now we claim that for every $a\in A$
\begin{equation}\label{e19}
\phi\circ\pi_A\circ\rho_A(a)=\phi(a).
\end{equation}
For the left hand side of (\ref{e19}) we have
\begin{equation*}
\begin{split}
\phi\circ\pi_A\circ\rho_A(a)=&\phi\circ\pi_A\circ(r_A\otimes r_A)\circ \rho_{A\times _TB}\circ q_A(a)\\
=&\phi\circ\pi_A\circ(r_A\otimes r_A)\circ \rho_{A\times _TB}(a,0)\\
\end{split}
\end{equation*}
Since $\pi_A\circ (r_A\otimes r_A)=r_A\circ \pi_{A\times_ TB}$, we have
\begin{equation}\label{e20}
\phi\circ\pi_A\circ\rho_A(a)=\phi\circ r_A\circ\pi_{A\times_ TB}\circ\rho_{A\times_ TB}(a,0).
\end{equation}
On the other hand, for every $(a,b)\in A\times_ TB$ we have 
\begin{equation}\label{e21}
\begin{split}
(\phi\circ r_A)(a,b)=&\phi(r_A(a,b))\\=&\phi(a+T(b))\\=&\phi(a)+\phi\circ T(b)\\
=&(\phi ,\phi\circ T)(a,b).
\end{split}
\end{equation}
The equations (\ref{e20}) and (\ref{e21}) imply that 
$$\phi\circ\pi_A\circ\rho_A(a)=(\phi ,\phi\circ T)\circ\pi_{A\times _TB}\circ \rho_{A\times_TB}(a,0).$$
The equation (\ref{e18}) says that for every $(a,b)\in A\times _TB$
we have $(\phi , \phi\circ T)\circ \pi_{A\times _TB}\circ\rho_{A\times _TB}(a,b)=\phi(a)+\phi\circ T(b).$
Hence we have 
\begin{equation*}
\begin{split}
\phi\circ \pi _A\circ\rho_A(a)=&(\phi, \phi\circ T)\circ\pi_{A\times _TB}\circ\rho_{A\times _TB}(a,0)\\=&\phi(a)+\phi\circ T(0)\\=&\phi(a).
\end{split}
\end{equation*}
Since $\rho_A$ is  a $A$-bimodule map, so $A$ is $\phi$-biprojective.

Conversely, let $A$ be  $\phi$-biprojective. 
Then there exists a bounded $A$-bimodule morphism $\rho: A\rightarrow A\hat{\otimes}A$ such that 
\begin{equation}\label{e30}
\phi\circ\pi_A\circ\rho_A=\phi.
\end{equation}
We define $\rho_{A\times _TB}:A\times _TB\rightarrow (A\times _TB)\hat{\otimes}(A\times _TB)$
by $\rho_{A\times _TB}=(q_A\otimes q_A)\circ \rho_A\circ r_A$, where $q_A:A\rightarrow (A\times _TB)$ is defined by $q_A(a) =(a,0)$ for every $a\in A$ and $r_A: A\times _TB\rightarrow A$ is defined by $r_A(a,b)=a+T(b)$ for every $(a,b)\in A\times_TB$.

Now we show that for every $(a,b)\in A\times _TB$ 
\begin{equation}\label{e31}
(\phi,\phi\circ T)\circ \pi _{A\times _TB}\circ\rho_{A\times _TB}(a,b)=(\phi, \phi\circ T)(a,b) =\phi(a)+\phi\circ T(b).
\end{equation}
To calculate the left hand side of (\ref{e31}) we have
\begin{equation*}
\begin{split}
(\phi,\phi\circ T)\circ \pi _{A\times _TB}\circ\rho_{A\times _TB}(a,b)
=&(\phi,\phi\circ T)\circ \pi _{A\times _TB}\circ(q_A\otimes q_A)\circ\rho _A\circ r_A(a,b)\\
=&(\phi,\phi\circ T)\circ \pi _{A\times _TB}\circ(q_A\otimes q_A)\circ \rho _A(a+T(b)).
\end{split}
\end{equation*}
Since $\pi_{A\times_TB}\circ (q_A\otimes q_A)=q_A\circ\pi_A$, we have 
\begin{equation}\label{e32}
\begin{split}
(\phi,\phi\circ T)\circ \pi _{A\times _TB}\circ\rho_{A\times _TB}(a,b)=&(\phi, \phi\circ T)\circ \pi_{A\times_TB}\circ(q_A\otimes q_A)\circ\rho_A(a+T(b))\\ =&(\phi,\phi\circ T)\circ q_A\circ\pi_A\circ\rho_A(a+T(b)).
\end{split}
\end{equation}
In contrast, for every $a\in A$ we have
\begin{equation}\label{e33}
\begin{split}
(\phi,\phi\circ T)\circ q_A(a)=&(\phi, \phi\circ T)(a,0)\\ =&\phi(a)+\phi\circ T(0)\\
=&\phi(a).
\end{split}
\end{equation}
The equation (\ref{e32}) and (\ref{e33}) follow that 
\begin{equation*}
\begin{split}
(\phi,\phi\circ T)\circ \pi_{A\times _TB}\circ\rho_{A\times _TB}(a,b)
=&(\phi, \phi\circ T)\circ q_A\circ\pi_A\circ\rho_A(a+T(b))\\ 
=&\phi\circ\pi_A\circ\rho_A(a+T(b)).
\end{split}
\end{equation*}
The equation (\ref{e30}) says that for every $a\in A$, we have $\phi\circ\pi_A\circ\rho_A(a)=\phi(a)$. 

Hence we have 
\begin{equation*}
\begin{split}
(\phi,\phi\circ T)\circ \pi_{A\times _TB}\circ\rho_{A\times _TB}(a,b)
=&\phi\circ \pi_A\circ\rho_A(a+T(b))\\ 
=&\phi(a+T(b))\\
=&\phi(a)+\phi\circ T(b).
\end{split}
\end{equation*}  
We note that $A$ is a Banach $A\times _TB$-bimodule with the following module actions 
$$(a,b)\cdot c=(a+T(b))c,\quad c\cdot(a,b)=c(a+T(b))\qquad ((a,b)\in A\times _TB, c\in A).$$
Since $\rho_{A\times_T B}$ is the composition of $A\times_T B$-bimodule maps, so $A\times_T B$
is $(\phi,\phi\circ T)$-biprojective.

(ii) Let $A\times _TB$ be $(0,\psi)$-biprojective. Then there exists a bounded 
$A\times _TB$-bimodule morphism $\rho_{A\times _TB} : A\times _TB\rightarrow (A\times _TB)\hat{\otimes}(A\times _TB)$
such that
\begin{equation}\label{e22}
(0,\psi)\circ\pi_{A\times _TB}\circ\rho_{A\times _TB}=(0,\psi).
\end{equation}
 We define $\rho_B: B\rightarrow B\hat{\otimes} B$
 by $\rho_B=(p_B\otimes p_B)\circ\rho_{A\times _TB}\circ q_B$, where 
 $q_B:B\rightarrow A\times _TB$ is defined by $q_B(b)=(0,b)$
 for every $b\in B$ and $p_B:(A\times _TB)\rightarrow B$
 is defined by $p_B(a,b)=b$
 for every $(a,b)\in A\times _TB$.
We show that for every $b\in B$
\begin{equation}\label{e23}
\psi\circ\pi_B\circ\rho_B(b)=\psi(b).
\end{equation} 
To calculate the left hand side of (\ref{e23}) we have
\begin{equation*}
\begin{split}
\psi\circ\pi_B\circ\rho_B(b)=&\psi\circ\pi_B\circ(p_B\otimes p_B)\circ \rho_{A\times _TB}\circ q_B(b)\\
=&\psi\circ\pi_B\circ(p_B\otimes p_B)\circ \rho_{A\times _TB}(0,b).
\end{split}
\end{equation*}
Since $\pi_B\circ(p_B\otimes p_B)=p_B\circ \pi_{A\times _TB}$,
therefore
\begin{equation}\label{e24}
\begin{split}
\psi\circ\pi_B\circ\rho_B(b)=&\psi\circ\pi_B\circ(p_B\otimes p_B)\circ\rho_{A\times_T B}(0,b)\\
=&\psi\circ p_B\circ\pi_{A\times_T B}\circ\rho_{A\times_T B}(0,b).
\end{split}
\end{equation}
In contrast,  for every $(a,b)\in A\times_T B$ we have 
\begin{equation}\label{e25}
\psi\circ p_B(a,b)=\psi(b)=(0,\psi)(a,b).
\end{equation}
The equation (\ref{e24}) and (\ref{e25}) follow that for every $(a,b)\in A\times_T B$
\begin{equation*}
\begin{split}
\psi\circ\pi_B\circ\rho_B(b)=&\psi\circ p_B\circ\pi_{A\times_T B}\circ\rho_{A\times_T B}(0,b)\\
=&(0,\psi)\circ\pi_{A\times_T B}\circ\rho_{A\times_T B}(0,b).
\end{split}
\end{equation*}
The equation (\ref{e22}) says that for every $(a,b)\in A\times_T B$
\begin{equation*}
\begin{split}
(0,\psi)\circ\pi_{A\times_T B}\circ\rho _{A\times_T B}(a,b)=&\psi(b).
\end{split}
\end{equation*}
Hence, we have 
$$\psi\circ\pi_B\circ\rho_B(b)=(0,\psi)\circ\pi_{A\times_T B}\circ\rho_{A\times_T B}(0,b)
=\psi(b).$$
Since $\rho_B$ is the composition of $B$-bimodule maps, so $B$ is $\psi$-biprojective.

Conversely, let $B$ be $\psi$-biprojective. Then there exists a bounded
 $B$-bimodule morphism $\rho_B: B\rightarrow B\hat{\otimes}B$ such that 
\begin{equation}\label{e26}
\psi\circ\pi_B\circ\rho_B=\psi.
\end{equation} 
We define $\rho_{A\times_T B}: A\times_T B\rightarrow (A\times_T B)\hat{\otimes}(A\times_T B)$
by $\rho_{A\times_T B}=(s_B\otimes s_B)\circ\rho_B\circ p_B$, where
$s_B: B\rightarrow (A\times_T B)$
is defined by 
$s_B(b)=(-T(b),b)$
for every $b\in B$ and $p_B:A\times_T B\rightarrow B$ is defined by $p_B(a,b)=b$
for every $(a,b)\in A\times_T B$. We show that for every $(a,b)\in A\times_T B$
\begin{equation}\label{e27}
(0,\psi)\circ\pi_{A\times_T B}\circ\rho_{A\times_T B}(a,b)=(0,\psi)(a,b)
=\psi(b).
\end{equation}
To calculate the left hand side of (\ref{e27}) we have
\begin{equation*}
\begin{split}
(0,\psi)\circ\pi_{A\times_T B}\circ\rho_{A\times_T B}(a,b)=&(0,\psi)\circ\pi_{A\times_T B}\circ(s_B\otimes s_B)\circ\rho_B\circ p_B(a,b)\\
=&(0,\psi)\circ\pi_{A\times_T B}\circ(s_B\otimes s_B)\circ\rho_B(b).
\end{split}
\end{equation*}
Since $\pi_{A\times_T B}\circ(s_B\otimes s_B)=s_B\circ \pi_B$, we have 
\begin{equation}\label{e28}
\begin{split}
(0,\psi)\circ\pi_{A\times_T B}\circ\rho_{A\times_T B}(a,b)=&(0,\psi)\circ\pi_{A\times_T B}\circ(s_B\otimes s_B)\circ\rho_B(b)\\
=&(0,\psi)\circ s_B\circ\pi_B\circ\rho_B(b).
\end{split}
\end{equation}
In contrast, for every $b\in B$ we have
\begin{equation}\label{e29}
(0,\psi)\circ s_B(b)=(0,\psi)(-T(b),b)
=\psi(b).
\end{equation}
The equations (\ref{e28}) and (\ref{e29}) follow that for every $(a,b)\in A\times_T B$
\begin{equation*}
\begin{split}
(0,\psi)\circ \pi_{A\times_T B}\circ \rho_{A\times_T B}(a,b)=&(0,\psi)\circ s_B\circ\pi_B\circ\rho_B(b)\\
=&\psi\circ\pi_B\circ\rho_B(b).
\end{split}
\end{equation*}
The equation (\ref{e26}) says that for every $b\in B$,
$\psi\circ\pi_B\circ\rho_B(b)=\psi(b)$. Hence we have
\begin{equation*}
(0,\psi)\circ\pi_{A\times_T B}\circ\rho_{A\times_T B}(a,b)=\psi\circ\pi_B\circ\rho_B(b)
=\psi(b).
\end{equation*}
We note that $B$ is a Banach $A\times_T B$-bimodule by the module actions
$$(a,b)\cdot d=bd,\quad\quad d\cdot (a,b)=db\quad\quad ((a,b)\in A\times_T B, d\in B).$$
Hence $\rho_B$
is $A\times_T B$-bimodule map.
Since $\rho_{A\times_T B}$ is the composition of $A\times_T B$-bimodule maps, so $A\times_T B$
is $(0,\psi)$-biprojective.
\end{proof}
\end{thm}
\section{approximate amenability and pseudo amenability}
In this section we investigate the approximate amenability and pseudo amenability of $A\times_TB$. 
We also provide a necessary  and sufficient conditions for  $(\phi,\phi\circ T)$-pseudo amenability and $(0,\psi)$-pseudo amenability
of $A\times_TB$. Next lemma is similar to \cite[Proposition 3.2]{ABT}, which we omit the proof. 
\begin{lem}\label{lem3.1}
Let $A$ and $B$ be Banach algebras and let  
$T:B\rightarrow A$ be an algebra homomorphism  with $\Vert T\Vert\leq 1$. Then 
\begin{enumerate}
\item[(i)] If $\lbrace (e_{\alpha},\eta_{\alpha})\rbrace$ is (bounded) weakly approximate identity for $A\times _TB$, then $\lbrace e_{\alpha}+T(\eta_{\alpha})\rbrace$ and $\lbrace \eta_{\alpha}\rbrace$ are (bounded) weakly approximate identities for $A$
and $B$, respectively.
\item[(ii)] If $\lbrace e_{\alpha}\rbrace$ and $\lbrace \eta_{\beta}\rbrace$ are (bounded) weakly approximate identities for $A$ and $B$, respectively, then $\lbrace(e_{\alpha}-T(\eta_{\beta}),\eta_{\beta})\rbrace$ is (bounded) weakly approximate identity for
$A\times_T B$.
\end{enumerate}
\end{lem}
We use an analogue version of the method used in the \cite[Theorem 4.1]{BHA} to prove the next theorem.
\begin{thm}\label{t3-2}
Let $A$ and $B$ be Banach algebras and 
let $A\times_ T B$ be approximately amenable for an algebra homomorphism $T: B\rightarrow A$ with $\Vert T\Vert\leq 1$. Then $A$ and $B$ are approximately amenable.
\end{thm}
\begin{proof}
Suppose that  $X$ is a Banach $A$-bimodule and  $d: A\rightarrow X^*$ is a bounded derivation. Simply we consider this derivation as a new derivation
$d:A\times \lbrace 0\rbrace\rightarrow X^*\times \lbrace 0\rbrace$. Moreover, $X\times \lbrace 0\rbrace$ is a Banach $A\times _T B$-
bimodule.
Let $\phi:A\times _T B\rightarrow A\times \lbrace 0\rbrace$ be a  map defined by
$$\phi(a,b)=(a+T(b),0)\qquad((a,b)\in A\times _T B).$$
Now we take $D:A\times _T B\rightarrow X^*\times \lbrace 0\rbrace$ defined by $D=d\circ \phi$. Then $D$ is a bounded derivation on $A\times _T B$.
Since $A\times _T B$ is approximately amenable, there exists a net $(y_\alpha)\subseteq X^*$ such that
$$D(a,b)=\lim ((a,b)\cdot (y_\alpha ,0)-(y_\alpha ,0)\cdot (a,b)) \qquad ((a,b)\in A\times _T B).$$
Therefore
$$d(a)=d(a,0)=D(a,0)=\lim ((a,0)\cdot (y_\alpha ,0)-(y_\alpha ,0)\cdot (a,0))=\lim (a\cdot y_\alpha -y_\alpha \cdot a).$$
which means that  $A$ is approximately amenable. 

Similarly, take $\phi :A\times _T B\rightarrow \lbrace 0\rbrace \times B$ defined by $\phi (a,b)=(0,b)$, 
it follows that $B$ is approximately amenable.
\end{proof}
In the following example we show that the converse of the previous theorem is not valid in general case. In fact we show that if $A$ and $B$ are approximately amenable,
then $A\times_TB$ is not necessarily approximately amenable for an algebra homomorphism $T:B\rightarrow A$
with $\Vert T\Vert\leq 1$.

\begin{example}
	Let $l^1$ denote the well-known space of complex sequences and 
	let $K(l^1)$ be the space of all compact operators on $l^1$. It is well known that the Banach algebra $K(l^1)$
	is amenable. We renorm $K(l^1)$ with the family of equivalent norm $\Vert\cdot\Vert^{k}$ such that its bounded left approximate identity  will be the constant 1 and its bounded right approximate identity will be $k+1$.

	Now if we consider $c_0$-direct-sum $A=\oplus_{k=1}^{\infty}(K(l^1),\Vert\cdot\Vert^{k})$, then $A$ has a bounded left approximate identity but no bounded right approximate identity, see \cite[page 3931]{GHAHR} for more details. Now by choosing $T=0$ we have $A\oplus A^{op}=A\times_T A^{op}$. Note that $A$ and $A^{op}$, the opposite algebra, are boundedly approximately 
	amenable \cite[Theorem 3.1]{GHAHR}, but sine $A\oplus A^{op}$ has no bounded approximate identity, it is not  approximately amenable \cite[Theorem 4.1]{GHAHR}.
\end{example}
\begin{thm}
Let $A$ and $B$ be Banach algebras and let $A\times _T B$ be  pseudo amenable  for an algebra homomorphism $T:B\rightarrow A$ 
with $\Vert T\Vert\leq 1$.
Then $A$ and $B$ are pseudo amenable.
\end{thm}
\begin{proof}
We define $\psi :A\times _T B\rightarrow B $ by $\psi (a,b)=b$, for every $(a,b)\in A\times _T B$.
This map is a continuous epimorphism from $A\times _T B$ onto $B$, hence $B$ is pseudo amenable \cite[Proposition 2.2]{GHAH}.
Similarly,  the map $\psi :A\times _T B\rightarrow A$ defined by $\psi (a,b)=a+T(b)$
is a continuous epimorphism from $A\times _T B$ onto $A$. Hence $A$ is pseudo amenable \cite[Proposition 2.2]{GHAH}.
\end{proof}
The concept of (right) $\phi$-pseudo amenability of Banach algebra $A$ is equivalent with the existence of an approximate $\phi$-mean. (An approximate $\phi$-mean is a not necessarily bounded net $\lbrace a_{\alpha}\rbrace\subseteq A$ such that
$$\phi(a_{\alpha})\rightarrow 1\quad
\hbox{and}
\quad \Vert aa_{\alpha}-\phi(a)a_{\alpha}\Vert\rightarrow 0$$
for all $a\in A$). 
\begin{thm}
Let $A$ and $B$ be Banach algebras and let $T:B\rightarrow A$
be an algebra homomorphism  with $\Vert T\Vert\leq 1$. Then 
\begin{enumerate}
\item[(i)] $A\times_TB$ is $(\phi ,\phi\circ T)$-pseudo amenable
if and only if $A$ is $\phi$-pseudo amenable,
\item[(ii)] $A\times_TB$ is $(0,\psi)$-pseudo amenable 
if and only if $B$ is $\psi$-pseudo amenable,
\end{enumerate}
where  $\phi\in \Delta (A)$ and $\psi\in \Delta(B)$.
\begin{proof}
(i) Let $A\times_TB$ be $(\phi ,\phi\circ T)$-pseudo amenable. Then
there exists a net $\lbrace(a_{\alpha},b_{\alpha})\rbrace \subseteq A\times_TB$
such that 
\begin{equation}\label{e61}
\begin{split}
(\phi ,\phi\circ T)(a_{\alpha},b_{\alpha})=&\phi(a_{\alpha})+\phi\circ T(b_{\alpha})\rightarrow 1
\end{split}
\end{equation}
and
\begin{equation}\label{e62}
\begin{split}
\Vert ((a,b)(a_{\alpha},b_{\alpha})-(\phi ,\phi\circ T)((a,b))(a_{\alpha},b_{\alpha}))\Vert\rightarrow 0\qquad ((a,b)\in A\times _TB).
\end{split}
\end{equation}
We show that  $\lbrace a_{\alpha}+T(b_{\alpha})\rbrace$ is an approximate $\phi$-mean
for $A$. 
By  equation (\ref{e62}) we have
\begin{equation}\label{e620}
\begin{split}
\Vert ((a,b)(a_{\alpha},b_{\alpha})-&(\phi ,\phi\circ T)((a,b))(a_{\alpha},b_{\alpha}))\Vert\\& =
\Vert (aa_{\alpha}+aT(b_{\alpha})+T(b)a_{\alpha},bb_{\alpha})-(\phi(a)+\phi\circ T(b))
(a_{\alpha},b_{\alpha})\Vert
\rightarrow 0.\end{split}
\end{equation}
Substituting $b=0$ in the equation (\ref{e620}), we have
\begin{equation*}
\begin{split}
\Vert (aa_{\alpha}+aT(b_{\alpha}),0)-(\phi(a)a_{\alpha},\phi(a)b_{\alpha})\Vert
=& \Vert(aa_{\alpha}+aT(b_{\alpha})-\phi(a)a_{\alpha},-\phi(a)b_{\alpha})\Vert\rightarrow 0.
\end{split}
\end{equation*}
Hence we have
$$\Vert aa_{\alpha}+aT(b_{\alpha})-\phi(a)a_{\alpha}\Vert\rightarrow 0$$
and 
$$\Vert\phi(a)b_{\alpha}\Vert\rightarrow 0.$$
In contrast, we have
\begin{equation}\label{e63}
\begin{split}
\Vert aa_{\alpha}+aT(b_{\alpha})-\phi(a)a_{\alpha}-\phi(a)T(b_{\alpha})\Vert
\leq &\Vert aa_{\alpha}+aT(b_{\alpha})-\phi(a)a_{\alpha}\Vert +\Vert -\phi(a)T(b_{\alpha})\Vert\\
\leq & \Vert aa_{\alpha}+aT(b_{\alpha})-\phi(a)a_{\alpha}\Vert +\Vert -\phi(a)b_{\alpha}\Vert\rightarrow 0.
\end{split}
\end{equation}
The equations (\ref{e63}) and (\ref{e61}) imply that 
 $\lbrace a_{\alpha}+T(b_{\alpha})\rbrace$
is an approximate $\phi$-mean for $A$.

Conversely,  let $A$ be $\phi$-pseudo amenable. Then there exists a net
$\lbrace a_{\alpha}\rbrace\subseteq A$ such that $\phi(a_{\alpha})\rightarrow 1$
and $\Vert aa_{\alpha}-\phi(a)a_{\alpha}\Vert\rightarrow 0$ for all $a\in A$.
We show that  $\lbrace(a_{\alpha},0)\rbrace$
is an approximate $(\phi ,\phi\circ T)$-mean for $A\times _TB$.
We have $(\phi ,\phi\circ T)(a_{\alpha},0)=\phi(a_{\alpha})\rightarrow 1$.
On the other hand for every $(a,b)\in A\times _TB$ we have
\begin{equation*}
\begin{split}
\Vert (a,b)(a_{\alpha},0)-(\phi(a)+\phi\circ T(b))(a_{\alpha},0)\Vert =& 
\Vert (aa_{\alpha}+T(b)a_{\alpha},0)-(\phi(a)a_{\alpha}+\phi\circ T(b)a_{\alpha},0)\Vert\\
=&\Vert (aa_{\alpha}- \phi(a)a_{\alpha}+T(b)a_{\alpha}-\phi\circ T(b)a_{\alpha},0)\Vert\\
\leq &\Vert aa_{\alpha}- \phi(a)a_{\alpha}\Vert +\Vert T(b)a_{\alpha}-\phi\circ T(b)a_{\alpha}\Vert\rightarrow 0.
\end{split}
\end{equation*}
Hence $\lbrace(a_{\alpha},0)\rbrace\subseteq A\times _TB$ is an approximate $(\phi,\phi\circ T)$-mean
for $A\times _TB$.

(ii) Suppose that $A\times_TB$ is $(0,\psi)$-pseudo amenable. 
Then there exists a net $\lbrace(a_{\alpha},b_{\alpha})\rbrace\subseteq A\times_TB$
such that for every $(a,b)\in A\times_TB$ we have
\begin{equation}\label{e66}
\begin{split}
(0,\psi)(a_{\alpha},b_{\alpha})=&\psi(b_{\alpha})\rightarrow 1
\end{split}
\end{equation}
and
\begin{equation}\label{e65}
\begin{split}
\Vert(a,b)(a_{\alpha},b_{\alpha})-(0,\psi)((a,b))(a_{\alpha},b_{\alpha})\Vert\rightarrow 0.
\end{split}
\end{equation}
By equation (\ref{e65}) we have
\begin{equation*}
\begin{split}
\Vert(a,b)(a_{\alpha},b_{\alpha})-(0,\psi)((a,b))(a_{\alpha},b_{\alpha})\Vert =&\Vert(aa_{\alpha}+T(b)a_{\alpha}+aT(b_{\alpha}),bb_{\alpha})-(\psi(b)a_{\alpha},\psi(b)b_{\alpha})\Vert \\
=& \Vert(aa_{\alpha}+T(b)a_{\alpha}+aT(b_{\alpha})-\psi(b)a_{\alpha},bb_{\alpha}-\psi(b)b_{\alpha})\Vert\rightarrow 0.
\end{split}
\end{equation*}
Hence we have
$$\Vert aa_{\alpha}+T(b)a_{\alpha}+aT(b_{\alpha})-\psi(b)a_{\alpha}\Vert\rightarrow 0$$
and 
\begin{equation}\label{e67}
\begin{split}
\Vert bb_{\alpha}-\psi(b)b_{\alpha}\Vert\rightarrow 0.
\end{split}
\end{equation}
The equations (\ref{e66}) and (\ref{e67}) follow that $\lbrace b_{\alpha}\rbrace\subseteq B$
is an approximate $\psi$-mean for B.

Conversely, suppose that $B$ is $\psi$-pseudo amenable. Then there
exists a net $\lbrace b_{\alpha}\rbrace\subseteq B$ such that for every $b\in B$
\begin{equation}\label{e68}
\begin{split}
\psi(b_{\alpha})\rightarrow 1
\quad\hbox{and}\quad
\Vert bb_{\alpha}-\psi(b)b_{\alpha}\Vert\rightarrow 0.
\end{split}
\end{equation}
We show that the net $(-T(b_{\alpha}),b_{\alpha})\subseteq A\times_TB$
is an approximate $(0,\psi)$-mean for $A\times_TB$.
Using (\ref{e68}) yields 
\begin{equation}\label{e70}
\begin{split}
(0,\psi)(-T(b_{\alpha}),b_{\alpha})=\psi(b_{\alpha})\rightarrow 1.
\end{split}
\end{equation}
Moreover, for every $(a,b)\in A\times_TB$ we have
\begin{equation*}
\begin{split}
\Vert (a,b)(-T(b_{\alpha}),b_{\alpha})-(0,\psi)((a,b))(-T(b_{\alpha}),b_{\alpha})\Vert =&
\Vert (-aT(b_{\alpha})+aT(b_{\alpha})\\ &\,\,-T(b)T(b_{\alpha}),bb_{\alpha})-(\psi(b)(-T(b_{\alpha}),b_{\alpha}))\Vert\\
=&\Vert (-T(b)T(b_{\alpha}),bb_{\alpha})-(-\psi(b)T(b_{\alpha}),\psi(b)b_{\alpha})\Vert\\
=&\Vert (-T(b)T(b_{\alpha})+\psi(b)T(b_{\alpha}),bb_{\alpha}-\psi(b)b_{\alpha})\Vert\\
=&\Vert-T(b)T(b_{\alpha})-\psi(b)T(b_{\alpha})\Vert +\Vert bb_{\alpha}-\psi(b)b_{\alpha}\Vert.
\end{split}
\end{equation*}
The equation (\ref{e68}) follows that 
$$T(bb_{\alpha})-\psi(b)T(b_{\alpha})\rightarrow 0 .$$
Hence for every $(a,b)\in A\times_TB$ we have
\begin{equation}\label{e71}
\begin{split}
\Vert (a,b)(-T(b_{\alpha}),b_{\alpha})-(0,\psi)((a,b))(-T(b_{\alpha}),b_{\alpha})\Vert\rightarrow 0.
\end{split}
\end{equation}
The equations (\ref{e71}) and (\ref{e70}) follow that  $\lbrace(-T(b_{\alpha}),b_{\alpha})\rbrace$ is an approximate 
$(0,\psi)$-mean for $A\times_TB$.
\end{proof}
\end{thm}
\section{Double centralizer  algebra}
In this section  we characterize double centralizer  algebra $M(A\times _T B)$ of $A\times_TB$ and as an application
we obtain a sufficient condition for approximate amenability of $A\times_TB$. 

If $A$ is a Banach algebra, then  the idealizer  $Q(A)$ of $A$ in
$(A^{**},\square)$ is defined by
$$Q(A)=\lbrace f_1\in A^{**}: x\square f_1 \ \ and\ \  f_1\square x\in A\ \  for \ \ every \ \ x\in A\rbrace,$$
see \cite{SAI} for more details.
 \begin{pro}\label{pro1}
Let $A$ and $B$ be Banach algebras and let $T:B\rightarrow A$
be an algebra homomorphism with $\Vert T\Vert\leq 1$.
Then $Q(A\times _T B)\cong Q(A)\times Q(B)$. 
\begin{proof}
By a similar method as in \cite[Theorem 3.1]{BHA}, we have $(A\times _T B)^{**}\cong A^{**}\times_{T^{**}}B^{**}$. We define a map $\phi : Q(A\times _T B)\rightarrow Q(A)\times Q(B)$ by $\phi (f_1,f_2)=(f_1+T^{**}(f_2),f_2)$
for every $ (f_1,f_2)\in Q(A\times _T B)\subseteq (A\times _T B)^{**}\cong A^{**}\times_{T^{**}}B^{**}$. Since
$(f_1,f_2)\in Q(A\times _T B)$ it follows that for every $(a,b) \in A\times _T B$, we have
$(f_1,f_2)\square (a,b)\in A\times _T B$. With a similar calculation  as in \cite[Theorem 3.1 ]{BHA}, we obtain that 
$$(f_1,f_2)\square (a,b)=(f_1\square a+T^{**}(f_2)\square a+ f_1\square T^{**}(b), f_2\square b).$$
Hence
\begin{equation}\label{2}
(f_1\square a+T^{**}(f_2)\square a+ f_1\square T^{**}(b), f_2\square b)\in A\times _T B.
\end{equation}
Substituting $b=0$ in  (\ref{2}) follows that $f_1\square a+T^{**}(f_2)\square a\in A$,
for all $a\in A$. Similarly $a\square f_1+a\square T^{**}(f_2)\in A$, for every $a\in A$. Hence
$f_1+T^{**}(f_2)\in Q(A)$. Clearly, equation (\ref{2}) shows that $f_2\square b$ and similarly $b\square f_2$ $\in B$ for all $b\in B$.
Therefore $f_2\in Q(B)$. Hence the map $\phi$ is well-defined.
Clearly the map $\phi$ is an isomorphism from $Q(A\times _T B)$ onto $Q(A)\times Q(B)$.
This completes the proof.
\end{proof}
\end{pro}
Recall that
$(A\times _T B)^{*}\cong (A^{*}\times B^{*}) $ \cite[Theorem 1.10.13]{MEG} and with a similar argument as in \cite[theorem 3.1]{BHA}, we have $(A\times _T B)^{**}\cong A^{**}\times_{T^{**}}B^{**}$. Hence we have $(A\times _T B)^{***}\cong A^{***}\times B^{***}$.
\begin{thm}
Let $A$ and $B$ be Banach algebras and let $T:B\rightarrow A$ be
an algebra homomorphism with $\Vert T\Vert\leq 1$.
If $A\times _T B$ has weakly
approximate identity
 $\lbrace (e_\alpha ,\eta _ \alpha)\rbrace$ such that $w-\lim (f,g)\square (e_\alpha ,\eta _ \alpha)=(f,g)$
for all $(f,g)\in (A\times _T B)^{**}$, then $M(A\times _T B)\cong M(A)\times M(B)$.
\end{thm}
\begin{proof}
	Suppose that the hypothesis holds.
Then using  \cite[Corollary 2]{SAI}, it implies that $Q(A\times _TB)\cong M(A\times _TB).$ 
Since $\lbrace (e_\alpha ,\eta _ \alpha)\rbrace$ is weakly
 approximate identity for $A\times _T B$, by Lemma \ref{lem3.1},
$\lbrace e_\alpha +T(\eta _ \alpha)\rbrace$
and $\lbrace \eta _{\alpha}\rbrace$ are weakly approximate identities for $A$ and
$B$, respectively. 

 Since $w-\lim (f,g)\square (e_\alpha ,\eta _ \alpha)=(f,g)$ for all $(f,g)\in (A\times _T B)^{**} $,
it follows that $(f^{***},g^{***})((f,g)\square (e_\alpha ,\eta _ \alpha))\rightarrow (f^{***},g^{***})(f,g)$
 for every $(f^{***},g^{***})\in (A\times _T B)^{***} $.
Hence for every $(f^{***},g^{***})\in  A^{***}\times B^{***} $
and $(f,g)\in  A^{**}\times_{T^{**}}B^{**}$ we have
\begin{equation}\label{3}
(f^{***},g^{***})(f\square e_{\alpha}+ f\square T(\eta _{\alpha})
+T(g)\square e_\alpha ,g\square \eta _{\alpha})\rightarrow
f^{***}(f)+g^{***}(g).
\end{equation}
Substituting $g=0$ in (\ref{3}), we obtain 
$f^{***}(f\square e_\alpha+f\square T(\eta _{\alpha}))\rightarrow f^{***}(f)$
for every $f^{***}\in A^{***}$ and $f\in A^{**}$.
 It follows that $w-\lim f\square (e_\alpha +T(\eta _ \alpha))=f $
for every $f\in A^{**}$. Using this fact and since $\lbrace e_\alpha +T(\eta _ \alpha)\rbrace$ is weakly approximate identity for $A$, 
 by \cite[Corollary 2]{SAI} we have $M(A)\cong Q(A)$. 
 
 Substituting $f^{***}=0$ in (\ref{3}),
 it follows that $g^{***}(g\square \eta _{\alpha})\rightarrow g^{***}(g)$
for every $g^{***}\in B^{***}$ and $g\in B^{**}$. 
Hence we have $w-\lim g\square \eta _ \alpha=g$
for every $g\in B^{**}$. Again using this fact and since $\lbrace \eta _{\alpha}\rbrace$ is weakly 
approximate identity for $B$,
 by \cite[Corollary 2]{SAI}, we have $M(B)\cong Q(B)$.
Combining these facts and Proposition \ref{pro1} yields
$$ M(A\times _TB) \cong Q(A\times _T B)\cong Q(A)\times Q(B)\cong M(A)\times M(B).$$
\end {proof}

\begin{thm}\label{thm4.4}
Let $A$ and $B$ be Banach algebras and let $T:B\rightarrow A$ be an 
algebra homomorphism with $\Vert T\Vert\leq 1$. If $A\times _TB$ has weakly bounded 
approximate identity, then $M(A\times _TB)\cong M(A)\times M(B)$.
\begin{proof}
Since   $A\times _TB$ has weakly bounded 
approximate identity,	\cite[Theorem 2]{SAI} shows that the map $\phi: Q(A\times _TB)\rightarrow M(A\times _TB)$
defined by $\phi(f_1,f_2)=(L_{(f_1,f_2)},R_{(f_1,f_2)})$ for every 
$(f_1,f_2)\in Q(A\times _TB)$ is onto.
 Set $$K_{A\times _TB}=\lbrace (f_1,f_2)\in A^{**}\times_{T^{**}}B^{**}: (g_1,g_2)\square (f_1,f_2)=0,\quad \forall(g_1,g_2)\in ( A^{**}\times_{T^{**}}B^{**})\rbrace.$$
By \cite[Theorem 1]{SAI} $\ker\phi=K_{A\times _TB}\cap Q(A\times _TB)$.
This implies that
\begin{equation}\label{e41}
\begin{split}
 \frac{Q(A\times _TB)}{K_{A\times _TB}\cap Q(A\times _TB)}\cong & M(A\times _TB),
\end{split}
\end{equation}
by a similar argument, we have 
\begin{equation}\label{e42}
\begin{split}
\frac{Q(A)}{K_A\cap Q(A)}\cong M(A)
\end{split}
\end{equation}
and 
\begin{equation}\label{e43}
\begin{split}
\frac{Q(B)}{K_B\cap Q(B)}\cong M(B),
\end{split}
\end{equation}
where $\cong$ denotes the algebra isomorphism.

Now we define $\psi:K_{A\times _TB}\rightarrow K_A\times K_B$
by $\psi(f_1,f_2)=(f_1+T^{**}(f_2),f_2)$. Clearly $\psi$ is an algebra isomorphism. 
Hence we have $K_{A\times _TB}\cong K_A\times K_B$
and by Proposition \ref{pro1} we have 
\begin{equation}\label{e44}
\begin{split}
Q(A\times _TB)\cong & Q(A)\times Q(B). 
\end{split}
\end{equation}
Therefore
\begin{equation}\label{e45}
\begin{split}
K_{A\times _TB}\cap Q(A\times _TB)\cong & (K_A \times K_B)\cap (Q(A)\times Q(B))\\
=&(K_A\cap Q(A))\times (K_B\cap Q(B)).
\end{split}
\end{equation}
The equations (\ref{e44}) and (\ref{e45}) imply  that
\begin{equation}\label{e46}
\begin{split}
\frac{Q(A\times _TB)}{K_{A\times _TB}\cap Q(A\times _TB)}\cong & \frac{Q(A)\times Q(B)}{(K_A\cap Q(A))\times (K_B\cap Q(B))}\\
\cong & \frac{Q(A)}{K_A\cap Q(A)}\times \frac{Q(B)}{K_B\cap Q(B)}.
\end{split}
\end{equation}
Combining  (\ref{e41}), (\ref{e42}) and (\ref{e43}) yields
$M(A\times _TB)\cong M(A)\times M(B)$.
\end{proof} 
\end{thm}
Now as an application of the previous theorem  we give a version of  inverse of  Theorem \ref{t3-2}, which is not valid in general case.

\begin{thm}
Let $A$ and $B$ be Banach algebras and
let $T:B\rightarrow A$ be an algebra homomorphism with $\Vert T\Vert\leq 1$.
Suppose that $A$ is approximately amenable, $A$ has a bounded approximate identity and $B$ is amenable. Then $A\times_TB$
is approximately amenable.
\end{thm}
\begin{proof}
Let $\lbrace e_{\alpha}; \alpha \in I\rbrace$ and $\lbrace\eta_{\beta}; \beta\in J\rbrace$
be bounded approximate identities for $A$ and $B$, respectively. By Lemma \ref{lem3.1}, 
$\lbrace(e_{\alpha}-T(\eta_{\beta}),\eta_{\beta})\rbrace$ is a bounded approximate identity for $A\times_TB$.
Hence by \cite[Proposition 1.8]{JOH}  $A\times_TB$ is approximate
amenable if and only if every bounded derivation from $A\times_TB$ into $X^*$ is approximate inner, where $X$ 
is a neo-unital Banach $A\times_TB$-bimodule.

Since $M(A\times_TB)\cong M(A)\times M(B),$ we can  define two maps $a\mapsto a\times \lbrace 1\rbrace$ and $b\mapsto \lbrace 1\rbrace\times b$, which are 
homomorphisms of $A$ and $B$ into two commuting subsets  of $M(A\times_TB)$,
 which are denoted
 by $A\times \lbrace 1\rbrace$ and $\lbrace 1\rbrace\times B$,
respectively.  With a similar argument as in \cite[Proposition 1.9]{JOH} we see that  $X$ is a Banach $M(A\times_TB)$-bimodule, so  it is an Banach $A$-bimodule and 
a Banach $B$-bimodule. Let $D\in \mathcal{Z}^1(A\times_TB,X^*)$. Then $D$ extends to an element of  
$\mathcal{Z}^1(M(A\times_TB),X^*)$  still denoted by $D$  and this gives an element $D_1$ of
$\mathcal{Z}^1(B,X^*)$ by restriction to $\lbrace 1\rbrace\times B$. Since $B$ is amenable, there 
exists $ y_0\in X^*$ such that $D_1(b)= ad_{ y_0}(b)= y_0\cdot b-b\cdot y_0$,
for every $b\in B$. Now let 
$\tilde D=D- ad_{ y_0}\in \mathcal{Z}^1(M(A\times_TB),X^*) $,
we have 
$\tilde D(1,b)=0$
for every $b\in B$, hence for every $a\in A$ and $b\in B$ we have
$$(1,b)\cdot \tilde D(a,1)=\tilde D(a,b)=\tilde D(a,1)\cdot (1,b),$$
which shows that  the range of $\tilde D\vert_{A\times \lbrace 1\rbrace}$
is in $\mathcal{Z}^0(\lbrace 1\rbrace\times B,X^*)=\lbrace x^*\in X^*: x^*\cdot (1,b)=(1,b)\cdot x^*\quad (b\in B)\rbrace$.
 Since $A$ is approximately amenable, $A$ is approximately contractible \cite[Theorem 2.1]{GHAHRA}. Hence there exists $\lbrace z_{\alpha}\rbrace\subseteq \mathcal{Z}^0(\lbrace 1\rbrace\times B, X^*)$
such that
$$\tilde D(a,1)=\lim z_{\alpha}\cdot (a,1)-(a,1)\cdot z_{\alpha}.$$
Since $\lbrace z_{\alpha}\rbrace\subseteq \mathcal{Z}^0(\lbrace 1\rbrace\times B, X^*)$,
we have $\tilde D=0$ on $\lbrace 1\rbrace\times B$.
Now we define derivation $\tilde{\tilde D}\in \mathcal{Z}^1(M(A\times_TB), X^*)$ by
 $\tilde{\tilde D}(a,b)=\tilde D(a,b)-(\lim z_{\alpha}\cdot (a,b)-(a,b)\cdot z_{\alpha})$ 
for every $(a,b)\in M(A\times_TB)\cong M(A)\times M(B)$.
we have $\tilde{\tilde D}=0$ on $(A\times \lbrace 1\rbrace)\cup(\lbrace 1\rbrace \times B)$.
Since $A\times_TB\hookrightarrow M(A\times_TB)\cong M(A)\times M(B)$,
for every $(a,b)\in A\times_TB$
$$\tilde{\tilde D}(a,b)=\tilde{\tilde D}((a,1)(1,b))=(a,1)\cdot\tilde{\tilde D}(1,b)+\tilde{\tilde D}(a,1)\cdot (1,b)=0,$$
that is, $\tilde D(a,b)-(\lim z_{\alpha}\cdot (a,b)-(a,b)\cdot z_{\alpha})=0$ for every $(a,b)\in A\times_TB$.

Hence 
\begin{equation*}
\begin{split}
D\vert_{A\times_TB}(a,b)=&\lim z_{\alpha}\cdot (a,b)-(a,b)\cdot z_{\alpha}+y_0\cdot (a,b)-(a,b)\cdot y_0\\
=&\lim (z_{\alpha}+y_0)\cdot (a,b)-(a,b)\cdot (z_{\alpha}+y_0).
\end{split}
\end{equation*}
Therefore $D\vert_{A\times_TB}$ is approximately inner. This completes the proof. 

\end{proof}

\end{document}